\documentclass[12pt]{amsart}

\usepackage{amsmath,amssymb,amsthm,commath}
\usepackage{enumitem,graphicx}
\title{The Commuting Graph of a Solvable $A$-Group}
\author{Rachel Carleton}
\address{Department of Mathematical Sciences, Kent State University, Kent, OH 44242}
\email{rcarlet3@kent.edu}
\author{Mark L. Lewis}
\address{Department of Mathematical Sciences, Kent State University, Kent, OH 44242}
\email{lewis@math.kent.edu}

\theoremstyle{plain}
\newtheorem{theorem}{Theorem}[section]
\newtheorem{lemma}[theorem]{Lemma}
\newtheorem{corollary}[theorem]{Corollary}

\newtheorem{hypothesis}[theorem]{Hypothesis}

\newcommand{\centerof}[1]{\mathbf{Z}(#1)}
\newcommand{\centralizer}[2]{\mathbf{C}_{#1}(#2)}
\newcommand{\normalizer}[2]{\mathbf{N}_{#1}(#2)}

\newcommand{\syl}[2]{\textnormal{Syl}_{#1}(#2)}

\begin{document}
	\maketitle
	\begin{abstract}
		Let $G$ be a finite group.  Recall that an $A$-group is a group whose Sylow subgroups are all abelian.  In this paper, we investigate the upper bound on the diameter of the commuting graph of a solvable $A$-group.  Assuming that the commuting graph is connected, we show when the derived length of $G$ is 2, the diameter of the commuting graph will be at most 4.  In the general case, we show that the diameter of the commuting graph will be at most 6.  In both cases, examples are provided to show that the upper bound of the commuting graph cannot be improved.
	\end{abstract}
	
	\section{Introduction}
	In this paper, all groups are finite.  Given a group $G$, we define the \textit{commuting graph} of $G$, denoted $\Gamma(G)$, to be the graph whose vertex set is the noncentral elements of $G$, and two distinct vertices are adjacent in $\Gamma(G)$ if and only if they commute in $G$.  A \textit{path} in a graph is an ordered list of vertices $x_1, x_2, \cdots, x_n$ where there is an edge between $x_i$ and $x_{i+1}$ for all $i$.  If any two vertices can be connected by a path, then the graph is said to be \textit{connected}.  The \textit{distance} between two vertices is the length of the shortest path between two vertices if such a path exists.  When the commuting graph is connected, we define the \textit{diameter} of the graph to be the maximum value of the set of distances between the vertices.  
	
	The commuting graphs of various families of groups have been studied by several authors in the past.  A common question that arises is when will the commuting graph of a group be connected, and if connected, what is the upper bound on the diameter.  In 2002, Segev and Seitz showed that the commuting graph of a classical simple group defined over a field of order greater than 5 is either disconnected or has diameter between 4 and 10 \cite{SS}.  In \cite{IranJafar}, Iranmanesh and Jafarzedah investigated the commuting graph of the symmetric and alternating groups, and determined that the commuting graph of such groups is either disconnected or has diameter at most 5.  In that same paper, Iranmanesh and Jafarzedah conjectured that the diameter of the commuting graph of any finite, non-abelian group could be bounded by a universal constant.  However, work by Giudici and Parker in \cite{GiuPark} provided a family of 2-groups that serve as a counterexample to that conjecture.  While the conjecture that there is a universal bound on the diameter of the commuting graph of any finite, non-abelian group turned out to be false, the conjecture does hold for groups with a trivial center.  In \cite{parkerSoluble}, Parker proved that a solvable group with a trivial center will be either disconnected or have diameter at most 8.  In that paper, he also provided an example of a family of groups that have a connected commuting graph with diameter exactly 8, showing that this bound is sharp.  Expanding on this work in \cite{parkerGen}, Parker and Morgan dropped the solvability hypothesis to show that for any group with a trivial center, the commuting graph will either be disconnected or have diameter at most 10.
	
	In \cite{REU}, the authors looked to expand the results on the commuting graph of a group with a trivial center to groups where $G' \cap \centerof{G}=1$.  They show that in this case, $\Gamma(G)$ will be connected if and only if $\Gamma(G/\centerof{G})$ is connected.  Furthermore, the authors prove that when connected, the two commuting graphs will have the same diameter.  Note that if $G' \cap \centerof{G}=1$, then $\centerof{G/\centerof{G}}$ is trivial.  Thus, the results of \cite{parkerGen} apply to $G/\centerof{G}$, and the authors of \cite{REU} conclude that $\Gamma(G/\centerof{G})$ will either be disconnected or have diameter at most 10.  If $G$ is solvable, the authors conclude that $\Gamma(G)$ will be disconnected or have diameter at most 8.
	
	One class of groups that satisfies the condition $G' \cap \centerof{G}=1$ is \textit{A-groups}, groups whose Sylow subgroups are all abelian.  Solvable $A$-groups where first discussed by Hall in some remarks at the end of \cite{Hall2} and written up formally by Taunt in \cite{Taunt}.  Since $A$-groups satisfy the property that $G' \cap \centerof{G}=1$, the results of \cite{REU} apply to $A$-groups.  In particular, the commuting graph of a solvable $A$-group will be disconnected or have diameter at most 8; however, the authors conjectured that the actual bound on the diameter of the commuting graph of a solvable $A$-group would be lower than 8.  In this paper, we will prove that the upper bound on the diameter of the commuting graph of an $A$-group is at most 6.  Since \cite{REU} establishes when the commuting graph of a solvable $A$-group will be disconnected, we will focus on the case when $G$ is a solvable $A$-group with a connected commuting graph.
	\begin{theorem}\label{main}
		Let $G$ be a solvable $A$-group such that $G/\centerof{G}$ is neither a Frobenius nor 2-Frobenius group.  Then, the diameter of the commuting graph of $G$ is at most 6. 
	\end{theorem}
	
	In Section 2, we provide some preliminary results, particularly on various properties of solvable $A$-groups, that will be needed to prove the bound on the diameter of the commuting graph. In Section 3, we focus on the case where the derived length of $G$ is 2.  We prove that in this case the diameter of the commuting graph will be at most 4.  In Section 4, we prove that the commuting graph of any solvable $A$-group has diameter at most 6.  Lastly, we provide an example of a solvable $A$-group with a connected commuting graph with diameter 6, showing that the bound of 6 cannot be further improved.

	\section{Preliminaries}
	
	Before we can prove Theorem 1.1, we first provide results about $A$-groups, given by Taunt in \cite{Taunt}.  The first result regards the intersection of the center and derived subgroup of $A$-groups, and applies to all $A$-groups, not just solvable ones.
	
	\begin{lemma}[\cite{Taunt} Theorem 4.1]\label{ACenter}
		Let $G$ be an $A$-group.  Then $G' \cap \centerof{G} = 1$.
	\end{lemma}

	With solvable $A$-groups, it can be useful to look at the terms of the derived series and the relative system normalizers of those terms.  A \textit{Sylow system} of a group $G$ is a set of Sylow subgroups, one for each prime dividing the order of $G$, such that each pair of Sylow subgroups permute.  Sylow systems were studied by Hall in \cite{HallSys}.  Given a Sylow system $\mathcal{S}$ of $G$, we can define its \textit{system normalizer} $M(G) = \{g \in G \mid gPg^{-1} = P, P \in \mathcal{S}\}$ to be the set of all elements in $G$ that normalize every Sylow subgroup in $\mathcal{S}$.  System normalizers were studied by Hall in \cite{Hall}.  If $G$ is part of a larger group $L$, then, given a Sylow system $\mathcal{S}$, of $G$, we can define the following group:
	\[M_L(G) = \{ \ell \in L \mid \ell P \ell^{-1} = P, P \in \mathcal{S} \} \]
	We call $M_L(G)$ a system normalizer of $G$ relative to $L$ or a relative system normalizer of $G$.  When $L=G$, we will call $M_L(G)$ an absolute system normalizer.
	
	In his paper on the construction of solvable groups \cite{Hall2}, Hall looked at constructing partially complemented extensions of a characterstic subgroup by a partial complement belonging to a characteristic class of conjugate subgroups.  In solvable subgroups, the relative system normalizers form such a characteristic class.  More specifically, Hall showed that if $H/N$ is any nilpotent factor of the (not necessarily solvable) group $G$, then $G = M_G(H)N$, where $M_G(H)$ is a system normalizer of $H$ relative to $G$.  Furthermore, when $G$ is solvable, Hall showed that one could choose $H$ and $N$ from adjacent terms in the lower nilpotent series.  The \textit{lower nilpotent series} is the series $G = L_0 \geq L_1 \geq \cdots \geq L_n \geq \cdots$, where $L_i$ is the smallest normal subgroup of $L_{i-1}$ such that $L_{i-1}/L_i$ is nilpotent. This series is also known as the lower Fitting chain or lower Fitting series. Note that in $A$-groups, nilpotent subgroups are necessarily abelian and so the lower nilpotent series coincides with the derived series, so we can take $H=G^{(i-1)}$ and $N=G^{(i)}$.  What makes Hall's construction theory especially useful for $A$-groups is that the relative system normalizers of the terms of the derived series are complements of the subsequent term in the derived series.
	
	\begin{lemma}[\cite{Taunt} Corollary 4.5]\label{ML}
		Let $G = G^{(0)} \trianglerighteq G' = G^{(1)} \trianglerighteq \cdots \trianglerighteq G^{(n)}=1$ be the Derived Series of $G$.  Let $M_G(G^{(i)})$ be a relative system normalizer for $G^{(i)}$ in $G$.  Then $G = M_G(G^{(i-1)})G^{(i)}$, and $M_G(G^{(i-1)})$ is a complement for $G^{(i)}$ in $G$.  In particular, if $M$ is any absolute system normalizer of $G$, then $M$ is a complement of $G'$ in $G$.
	\end{lemma}

	When bounding the diameter of the commuting graph of a solvable $A$-group, we will search for paths through the Fitting subgroup.  In $A$-groups, we have that the Fitting subgroup is the direct product of the centers of the terms of the derived series.
	
	\begin{lemma}[\cite{Taunt} Theorem 5.4]\label{FitGroup}
		Let $G$ be a solvable $A$-group with derived length $n$.  Then \[F(G) = Z(G) \times Z(G') \times \cdots \times Z(G^{(n-1)}).\]
	\end{lemma}
	To prove the bound of 6 on the diameter of the commuting graph of an $A$-group, we will also need the following result on Frobenius groups.  Note that the next lemma applies to more than just $A$-groups.
	\begin{lemma}[\cite{Isaacs} Theorem 6.4]\label{Frob1}
		Let $N$ be a normal subgroup of a finite group $G$, and suppose that $A$ is a complement for $N$ in $G$.  The following are equivalent:
		\begin{enumerate}
			\item $G$ is a Frobenius group with kernel $N$.
			\item $\centralizer{G}{a} \leq A$ for all nontrivial $a \in A$.
			\item $\centralizer{G}{n} \leq N$ for all nontrivial $n \in N$.
		\end{enumerate}
	\end{lemma}
	Lastly, we will need the following result, which delivers elements with nontrivial centralizers.  Again, note this result applies to more than just $A$-groups.
	\begin{lemma}[\cite{Aschb} Theorem 36.2]\label{Asch}
		Let $p,q$, and $r$ be distinct primes, $X$ be a group of order $r$ that acts faithfully on a $q$-group $Q$ and $V$ be a faithful $\text{GF}(p)XQ$-module.  Additionally if $q=2$ and $r$ is a Fermat prime, assume $Q$ is abelian.  Then, $\centralizer{V}{X}\neq 0$.
	\end{lemma}
	
	The paper \cite{REU} establishes for $A$-groups that $\Gamma(G)$ will be connected if and only if $\Gamma(G/\centerof{G})$ is connected. When $G/\centerof{G}$ is solvable, $\Gamma(G/\centerof{G})$ is disconnected if and only if $G/\centerof{G}$ is a Frobenius or 2-Frobenius group.  A 2-Frobenius group is a group $X$ with normal subgroups $H$ and $K$ such that $H$ is a Frobenius group with kernel $K$ and $X/K$ is a Frobenius group with kernel $H/K$.  Since we are interested in finding the upper bound of the diameter of a connected commuting graph, we will focus our attention on the case when $G$ is a solvable $A$-group with a connected commuting graph.  That is when $G/\centerof{G}$ is neither a Frobenius nor 2-Frobenius group.  We will thus use the following hypothesis in many of the proofs in Sections 3 and 4.
	\begin{hypothesis}\label{hyp2}
		Let $G$ be a solvable, nonabelian $A$-group such that $G/\centerof{G}$ is neither a Frobenius nor 2-Frobenius group.
	\end{hypothesis}
	
	From \cite{REU}, we also have for $A$-groups that, when connected, the diameter of $\Gamma(G)$ and $\Gamma(G/\centerof{G})$ will be the same. Using this result, we may additionally narrow our focus down to solvable $A$-groups with a trivial center.  For if $\centerof{G}\neq 1$, then we can look instead at the commuting graph of $G/\centerof{G}$, which is a group with a trivial center.  Note that when $\centerof{G}$ is trivial and $G$ has a connected commuting graph, we then get that $G$ is neither a Frobenius nor 2-Frobenius group. 
	
	Lastly, we will establish some notation that will be used throughout the rest of the paper.  For any group $G$, let $G^{\#}$ represent the set of nontrivial elements in $G$.    Let $\pi(G)$ represent the set of prime divisors of the order of $G$.  If two noncentral elements commute, we will write $x \sim y$ as such notation helps us to count the length of a path between $x$ and $y$ in $\Gamma(G)$.  If $x$ and $y$ are noncentral elements, then $d(x,y)$ denotes the distance between $x$ and $y$ in the commuting graph.
	
	\section{Case When $G'$ is Abelian}
	Before investigating $A$-groups in general, we will first investigate the case where $G$ has derived length 2 and is thus metabelian.  In these groups, we get an even lower bound on the diameter of the commuting graph. 
	
	\begin{lemma}\label{need}
		Let $G$ be a solvable $A$-group of derived length 2 with a trivial center such that $G'$ is not a Hall subgroup.  Let $g$ be an element of $G$ and $p$ a prime dividing the order of $g$.  If the $p$-part of $g$ does not lie in $G'$ and does not lie any absolute system normalizer of $G$, then $\centralizer{G'}{g} \neq 1$ and there exists an absolute system normalizer $M$ of $G$ such that $\centralizer{M}{g} \neq 1$.
	\end{lemma}
	
	\begin{proof}
		Let $G$ be a solvable $A$-group of derived length 2 with a trivial center such that $G'$ is not a Hall subgroup.  From Lemma \ref{ML}, we have that $G=MG'$, where $M$ is an absolute system normalizer of $G$ and $M \cap G' = 1$. Since $\centerof{G}=1$, note that $F(G)=G'$ by Lemma \ref{FitGroup}.  Let $g$ be an element of $G$ such that the $p$-part of $g$ does not lie in $G'$ for some prime $p$.  Denote the $p$-part of $g$ by $g_p$.  If $g_p$ does not lie in any absolute system normalizer of $M$, then we have that $p$ divides both $\lvert G' \rvert$ and $\lvert M \rvert$.  Let $P$ be a Sylow $p$-subgroup of $G$ containing $g_p$. We have that $P = M_p \times O_p(G)$, where $M_p$ is the Sylow $p$-subgroup of an appropriately chosen conjugate of $M$.  Since $g_p \notin M_p$, we have that $g_p = m_pf_p$, with $m_p \in M_p$ and $f_p \in O_p(G)$.  Note that because $m_p$ and $f_p$ lie in the same Sylow $p$-subgroup of $G$, they will commute.  We will show that $g \sim m_p$ and $g \sim f_p$.
		
		If $g = g_p$, then $g$ is a $p$-element. Furthermore, $g$ lies in the same Sylow subgroup as $m_p$ and $f_p$, so we have that $g \sim m_p$ and $g \sim f_p$.  So suppose that the order of $g$ is divisible by at least 2 distinct primes.  We can write $g = g_p g_{q_1}g_{q_2} \cdots g_{q_r}$, where $g_{q_i}$ represents the $q_i$-part of $g$ for distinct primes $q_1, q_2, \cdots,q_r$. To show that $g \sim m_p$ and $g\sim f_p$, we will show that $g_{q_i} \sim m_p$ and $g_{q_i} \sim f_p$ for each $i$.
		
		We can write $g_{q_i} = m_{q_i}f_{q_i}$, where $m_{q_i}$ lies in a conjugate of $M$ and $f_{q_i} \in G'$.  If $m_{q_i} =1$, then $g_{q_i} = f_{q_i} \in G'$.  As $G'$ is abelian, we have that $f_p \sim f_{q_i}$. 
		Therefore, since $g_p \sim g_{q_i}=f_{q_i}$, we get that $f_{q_i}m_pf_p = m_pf_pf_{q_i} = m_pf_{q_i}f_p$.  Hence, $f_{q_i} \sim m_p$, and we are done.  Therefore, suppose that $m_{q_i} \neq 1$.  We then get the following
		\begin{align*}
			g_pg_{q_i} &= g_{q_i}g_p \\ m_pf_pm_{q_i}f_{q_i} &= m_{q_i}f_{q_i}m_pf_p \\ m_pf_pf_{q_i}m_{q_i} &= m_{q_i}f_{q_i}f_p m_p \\ m_pf_{q_i}f_pm_{q_i} &= m_{q_i}f_{q_i}f_p m_p \\ m_{q_i}^{-1}m_p \left( f_{q_i}f_p \right) &= \left( f_{q_i}f_p \right) m_{p}m_{q_i}^{-1}
		\end{align*}
		Note that we have not assumed that $m_p$ and $m_{q_i}$ lie in the same conjugate of $M$, so we must show that these elements commute. To do so, let $K$ be a Hall $\{p,q_i\}$-subgroup of $G$ containing $g_pg_{q_i}$. Let $K_0$ be the intersection of $K$ with an appropriately chosen conjugate of $M$ and let $K_1 = K \cap G'$. Let $M_p$ and $M_{q_i}$ be the Sylow $p$- and $q_i$-subgroups of $K_0$ respectively, and let $F_p$ and $F_q$ be the Sylow $p$- and $q$-subgroups of $K$ respectively.  Note that $F_p = O_p(G)$ and $F_q = O_q(G)$.  We will show that the number of conjugates of $K_0$ in $K$ is the number of conjugates of $M_p$ times the number of conjugates of $M_{q_i}$.  
		
		First, we will show that $\normalizer{K}{M_p} = \normalizer{F_{q_i}}{M_p} F_pK_0$.  Note that as $K_0$ is abelian, it will centralize $M_p$.  Similarly, as the elements of $F_p$ centralize every $p$-element in $G$, we have that $F_p$ centralizes $M_p$.  Thus, $\normalizer{F_{q_i}}{M_p}F_pK_0 \leq \normalizer{K}{M_p}$.  Thus, we need to show the reverse containment.  Let $x \in \normalizer{K}{M_p}$.  We can write $x = x_1x_0$ for some $x_1 \in K_1, x_0 \in K_0$.  Furthermore, note that $x_1 = k_{q_i}k_{p}$, where $k_{q_i}\in F_{q_i},k_{p}\in F_p$.  Then, we have that
		\begin{align*}
			M_p = xM_px^{-1} &= \left(k_{q_i}k_px_0\right)M_p \left(k_{q_i}k_px_0\right)^{-1} \\ &= k_{q_i}k_px_0 M_p x_0^{-1} k_p^{-1}k_{q_i}^{-1} \\ &= k_{q_i}k_pM_pk_p^{-1}k_{q_i}^{-1} \\ &= k_{q_i}M_p k_{q_i}^{-1}
		\end{align*}
		Hence, $k_{q_i} \in \normalizer{F_{q_i}}{M_p}$.  Therefore, $x = k_{q_i}k_px_0 \in \normalizer{F_{q_i}}{M_p} F_pK_0$.  Through a similar argument, we can show that $\normalizer{K}{M_{q_i}} = \normalizer{F_p}{M_{q_i}}F_{q_i}K_0$ and that $\normalizer{K}{K_0} = \normalizer{K_1}{K_0}K_0$.
		
		Next, we want to show that $\normalizer{K_1}{K_0} = \normalizer{K_1}{M_p} \cap \normalizer{K_1}{M_{q_i}}$.  To do this, first suppose $x \in \normalizer{K_1}{K_0}$.  Then, we have that $xM_px^{-1} \subseteq K_0$, so $x$ will send $M_p$ to one of its conjugates in $K_0$.  However, as $K_0$ is abelian (as it is a subgroup of a conjugate of $M$, which is abelian), we have that $M_p$ is the only Sylow $p$-subgroup of $K_0$.  Hence, $xM_px^{-1} = M_p$.  Similarly, we have that $xM_{q_i}x^{-1} = M_{q_i}$.  Therefore, $x \in \normalizer{K_1}{M_p} \cap \normalizer{K_1}{M_{q_i}}$.  Now, suppose $x \in \normalizer{K_1}{M_p} \cap \normalizer{K_1}{M_{q_i}}$.  Then, we have that $xM_px^{-1}=M_p$ and $xM_{q_i}x^{-1} = M_{q_i}$.  Thus, we have that $xK_0x^{-1} = x(M_pM_{q_i})x^{-1} = xM_px^{-1}xM_{q_i} x^{-1} = M_pM_{q_i} = K_0$.  Hence, $x\in\normalizer{K_1}{K_0}$.  Thus, we have that $\normalizer{K_1}{K_0} = \normalizer{K_1}{M_p}\cap\normalizer{K_1}{M_{q_i}} = F_p\normalizer{F_{q_i}}{M_p} \cap \normalizer{F_p}{M_{q_i}}F_{q_i} = \normalizer{F_p}{M_{q_i}}\normalizer{F_{q_i}}{M_p}$.
		
		Using this, we see that
		\[\lvert K : \normalizer{K}{M_p} \rvert = \frac{\lvert K \rvert}{\lvert K_0 \rvert \lvert F_p \rvert \lvert \normalizer{F_{q_i}}{M_p} \rvert} = \frac{\lvert F_{q_i} \rvert}{\lvert \normalizer{F_{q_i}}{M_p} \rvert}\]
		\[\lvert K : \normalizer{K}{F_{q_i}} \rvert = \frac{\lvert K \rvert}{\lvert K_0 \rvert \lvert F_{q_i} \rvert \lvert \normalizer{F_p}{M_{q_i}} \rvert} = \frac{\lvert F_p \rvert}{\lvert \normalizer{F_p}{M_{q_i}} \rvert}\]
		\[\lvert K : \normalizer{K}{K_0)} \rvert = \frac{\lvert K \rvert}{\lvert K_0 \rvert \lvert \normalizer{K_1}{K_0} \rvert} = \frac{\lvert K_1 \rvert}{ \lvert \normalizer{K_1}{K_0} \rvert} = \frac{\lvert F_p \rvert \lvert F_{q_i} \rvert}{\lvert \normalizer{F_p}{M_{q_i}} \rvert \lvert \normalizer{F_{q_i}}{M_p} \rvert}\]
		Thus, the number of conjugates of $K_0$ in $K$ is exactly the number of conjugates of $M_p$ times the number of conjugates of $M_{q_i}$.  This means that each conjugate of $M_p$ commutes with each conjugate of $M_{q_i}$ to form all the distinct conjugates of $K_0$, meaning that each conjugate of $M_p$ centralizes each conjugate of $M_{q_i}$.  Thus, we have that $m_p \sim m_{q_i}$.
		
		This gives us that $m_{q_i}^{-1}m_p \sim f_{q_i}f_p$.  Therefore, $\left(m_{q_i}^{-1}m_p \right)^{n_1} \sim \left(f_{q_i}f_p \right)^{n_2}$ for any $n_1,n_2 \in \mathbb{Z}$. Therefore, we have that $m_{q_i} \sim f_p$ and $f_{q_i} \sim m_p$.  Therefore, $g_{q_i} \sim m_p$ and $g_{q_i} \sim f_p$.  Since this holds for all $i$, we therefore have that $g \sim m_p$ and $g \sim f_p$.  Hence, $\centralizer{M}{g} \neq 1$ and $\centralizer{G'}{g} \neq 1$.
	\end{proof}
	
	\begin{lemma}\label{d2}
		Assume Hypothesis \ref{hyp2} and suppose that $\centerof{G}=1$. Furthermore, suppose that the derived length of $G$ is 2.  Then $\Gamma(G)$ is connected of diameter at most 4.
	\end{lemma}
	
	\begin{proof}
		Let the derived length of $G$ be 2.  Thus, $G$ is metabelian and $G'$ is abelian.  From Lemma \ref{ML}, we have that $G=MG'$, where $M$ is an absolute system normalizer of $G$, and $M \cap G' = 1$.  As $M \cong G/G'$, we have that $M$ is abelian.  Additionally, since $\centerof{G}=1$, we have that $F(G) = G'$ by Lemma \ref{FitGroup}. 
		
		Since $G$ is not a Frobenius group, there exists $x^* \in M^{\#}$ such that $\centralizer{G'}{x^*} \neq 1$.  Consider $\langle x^* \rangle G'$.  Let $g \in G^{\#}$ and note that we can write $g = mg'$, where $m \in M$ and $g \in G'$.  We have that $g \langle x^* \rangle G' g^{-1} = (mg')\langle x^* \rangle G' (mg')^{-1} = m \langle x^* \rangle G' m^{-1}$ as $g' \in G' \leq \langle x^* \rangle G'$.  Furthermore, as $m \in M$ and $G' \trianglelefteq G$, we have that $m \langle x^* \rangle G' m^{-1} = \langle x^* \rangle G'$.  Therefore, $\langle x^* \rangle G' \trianglelefteq G$.  Hence, as $M$ intersects $\langle x^* \rangle G'$ nontrivially, we have that every conjugate of $M$ intersects $\langle x^* \rangle G'$ nontrivially.  Since $\centralizer{G'}{x^*} \neq 1$, we have that $\centerof{\langle x^* \rangle G'}\neq 1$.  So fix $c^* \in \centerof{\langle x^* \rangle G'}^{\#}$.  We will show that $d(g,c^*) \leq 2$ for all $g \in G^{\#}$.
		
		Let $g \in G^{\#}$.  If there exists a positive integer $n < o(g)$ such that $g^n \in G'$, we have that $g \sim g^n \sim c^*$.  Hence, $d(g,c^*) \leq 2$.  Suppose that there does not exist such an integer $n$.  In particular, this gives us that $g \notin G'$.  If $g$ lies in some conjugate of $M$, we have that $g \sim x \sim c^*$, where $x$ lies in the intersection of $\langle x^* \rangle G'$ with an appropriately chosen conjugate of $M$.  Hence, $d(g,c^*) \leq 2$.  That leaves the case where $g$ is neither an element of $G'$ or of a conjugate of $M$.  Hence, there exists a prime $p$ such that the $p$-part of $g$ is neither an element of $G'$ or a conjugate of $M$.  Note in this case, $G'$ is not a Hall subgroup of $G$.  Therefore, $g$ satisfies the hypothesis of Lemma \ref{need}, and so $\centralizer{G'}{g} \neq 1$.  Hence, we have that $g \sim f \sim c^*$, where $f \in \centralizer{G'}{g}^{\#}$.  Therefore, $d(g,c^*) \leq 2$.  Thus, we have that for any $g \in G$, $d(g,c^*) \leq 2$.  This means that the diameter of the commuting graph of $G$ will be at most 4.
	\end{proof}
	
	Since we have proven that the diameter of the commuting graph of an $A$-group with a trivial center will be at most 4, we will now prove that any $A$-group that has a connected commuting graph will have diameter at most 4.
	
	\begin{lemma}\label{d2.2}
		Let $G$ satisfy Hypothesis \ref{hyp2}, and assume that $G$ has derived length 2.  Then, $\Gamma(G)$ is connected with diameter at most 4.
	\end{lemma}
	
	\begin{proof}
		Let $G$ be a solvable $A$-group satisfying Hypothesis \ref{hyp2} and let $Z=\centerof{G}$.  We then have that $G/Z$ is neither a Frobenius nor a 2-Frobenius group.  Furthermore, the commuting graphs of $G$ and $G/Z$ have the same diameter.  Recall that Lemma \ref{ACenter} gives us that $G' \cap Z = 1$.  Thus, we have that $(G/Z)' = G'Z/Z \cong G'$. Therefore, as $G$ is metablian and thus $G'$ abelian, we get that $(G/Z)'$ is abelian. Therefore, the derived length of $G/Z$ is 2.  From the previous lemma, we get that the diameter of $\Gamma(G/Z)$ is at most 4.  Hence, the diameter of $\Gamma(G)$ is at most 4.
	\end{proof}
	
	Note that SmallGroup(60,7) from the GAP Small Groups Library is an example of an $A$-group of derived length 2 with a connected commuting graph of diameter 4, so the bound in Lemma \ref{d2} cannot be improved.  The following two corollaries apply the results of Lemma \ref{d2.2} to particular classes of $A$-groups that have derived length 2.
	\begin{corollary}
		Suppose that $G$ satisfies Hypothesis \ref{hyp2} and that the order of $G$ is $p^aq^b$ for distinct primes $p$ and $q$.  Then $\Gamma(G)$ is connected of diameter at most 4.
	\end{corollary}
	
	\begin{proof}
		By Theorem 8.3 from \cite{Taunt}, we have that the derived length is at most 2.  As $G$ is non-Abelian, we have that the derived length is exactly 2.  Thus, by Lemma \ref{d2.2}, we have that $\Gamma(G)$ is connected of diameter at most 4.
	\end{proof}
	
	\begin{corollary}
		Suppose that $G$ is a group of cube-free odd order such that $G/\centerof{G}$ is neither a Frobenius nor 2-Frobenius group.  Then $\Gamma(G)$ is connected of diameter at most 4.
	\end{corollary}
	
	\begin{proof}
		Suppose that $G$ is a group of cube-free order and that the order of $G$ is odd.  Note that as $G$ has cube-free order, $G$ is an $A$-group.  By Theorem 8.4 from \cite{Taunt}, the derived length of $G$ is 2.  Hence, by Lemma \ref{d2.2}, we have that $\Gamma(G)$ is connected of diameter at most 4.
	\end{proof}
	
	\section{Main Result}
	We will now turn our attention to proving the bound on the diameter of the commuting graph of any solvable $A$-group.  To do so, we will assume that the diameter of the commuting graph is at least 7 and proceed by contradiction.  Let $F = F(G)$, $Z = \centerof{G}$, and $J$ be the preimage of $F(G/F)$.  From \cite{parkerSoluble}, we have that a solvable group with a trivial center that has a commuting graph of diameter 7 or higher has a fairly restricted structure.  The first thing we can say about such a group is that the Fitting subgroup is the centralizer of a minimal normal subgroup.
	
	\begin{lemma}[\cite{parkerSoluble} Lemma 3.4]\label{p4}
		Let $G$ be a solvable group with a trivial center and let $V$ be a minimal normal subgroup of $G$.  If the diameter of $\Gamma(G)$ is at least 7, then $F = \centralizer{G}{V}$.
	\end{lemma}
	
	Furthermore, we get that the Fitting subgroup is a Hall subgroup of $J$.
	\begin{lemma}[\cite{parkerSoluble} Lemma 3.5]\label{p5} 
		Let $G$ be any solvable group with a trivial center.  If the diameter of $\Gamma(G)$ is at least 7, then $\lvert J /F \rvert$ is coprime to $\lvert F \rvert$.
	\end{lemma}

	The following lemma is adapted from Lemma 3.9 of \cite{parkerSoluble}, with the statement and proof modified to better work with $A$-groups.
	
	\begin{lemma}\label{p9}
		Suppose $G$ satisfies Hypotheses \ref{hyp2} and that $\centerof{G}=1$.  Assume that the diameter of $\Gamma(G) \geq 7$ and that $V$ is a minimal normal subgroup of $G$.  Let $g \in G$ and $n \in \mathbb{Z}$ such that $o(g^n) = p$ for some prime $p$.  If $\centralizer{V}{g^n} =1$, then $g^n \in J$.
	\end{lemma}
	
	\begin{proof}
		Let $g\in G$ such that $o(g^n) = p$ for some prime $p$ and that $\centralizer{V}{g^n}=1$.  Suppose $g^n \notin J$.  As $G/F$ is solvable and $J/F = F(G/F)$, we have that there exists a prime $t$ such that the Sylow $t$-subgroup of $J/F$ is not centralized by $g^n$.  Since $G$ is an $A$-group (and therefore $G/F$ is also an $A$-group), we have that $p \neq t$.  Choose $T \in \syl{t}{J}$ to be $\langle g^n \rangle$ invariant.  Note that $T\langle g^n \rangle$ acts faithfully on  $V$ by Lemma \ref{p4}.  However, as $\centralizer{V}{g^n} = 1$, Lemma \ref{Asch} gives us a contradiction.  Therefore, $g^n \in J$.
	\end{proof}

	We are now ready to prove that an $A$-group with a trivial center that is neither Frobenius nor 2-Frobenius will have a commuting graph with diameter at most 7. 
	
	\begin{theorem}\label{dn}
		Let $G$ satisfy Hypotheses \ref{hyp2} and suppose that $\centerof{G}=1$.  Then, $\Gamma(G)$ is connected of diameter at most 6.
	\end{theorem}
	
	\begin{proof}
		Suppose that $G$ satisfies Hypotheses \ref{hyp2} and that $\centerof{G}=1$.  Since $G$ is neither Frobenius nor 2-Frobenius, we know that $\Gamma(G)$ is connected by the main result of \cite{parkerSoluble}.  Therefore, we need only prove that the upper bound of the diameter of the commuting graph is 6.  We proceed by contradiction and assume that $\Gamma(G)$ has diameter at least 7. If we can show that every element in $G$ has a distance of at most 3 from some specific element $c^* \in F^{\#}$, then we will get that the diameter of $\Gamma(G)$ is at most 6, giving us our desired contradiction.  
		
		First, assume that $J$ is not a Frobenius group.  From Lemma \ref{p5}, we have that $\lvert J/F \rvert$ and $\lvert F \rvert$ are coprime.  Thus, $F$ is a normal Hall subgroup of $J$ and so has a complement in $J$.  Let $K$ be a complement to $F$ in $J$.  Note that as $K \cong J/F$, we have that $K$ is abelian.  Since $J$ is not Frobenius, we have that there exists $k^* \in K^{\#}$ such that $C^* = \centralizer{F}{k^*} \neq 1$.  Consider $\normalizer{G}{C^*}$.  We have that as $F$ is abelian and $C^* \leq F$ that $F$ normalizes $C^*$.  Since $F \trianglelefteq G$, we have that $xcx^{-1} \in F$ for all $c \in F, x \in K$.  Therefore, since $(xcx^{-1})k^* = xck^*x^{-1} = xk^*cx^{-1} = k^*(xcx^{-1})$ for all $x \in K, c \in C^*$, we have that $xcx^{-1} \in C^*$ and so $K$ normalizes $C^*$.  Thus, $ \normalizer{G}{C^*} \geq  K  F  =  J$ and so $J$ normalizes $C^*$.  Hence, every conjugate of $K$ lies in $\normalizer{G}{C^*}$.  In particular, as $K$ intersects $\centralizer{G}{C^*}$ nontrivially, every conjugate of $K$ intersects $\centralizer{G}{C^*}$ nontrivially.
		
		Fix $c^* \in \left(C^*\right)^{\#}$.  Let $g \in G$.  If we can find a positive integer $n <o(g)$ so that $\centralizer{F}{g^n}\neq 1$, then we have that
		\[g \sim g^n \sim f \sim c^*,\]
		for some $f \in \centralizer{F}{g^n}^{\#}$.  Thus, we must assume that $\centralizer{F}{g^n} = 1$ for every positive integer $n < o(g)$.  In particular, when $g^n$ has prime order, we have that $g^n \in J$ by Lemma \ref{p9}.  As $\centralizer{F}{g^n}=1$, we have that $g^n$ has order coprime to the order of $F$.  Therefore, $g^n$ lies is some Hall $\pi(J/F)$-subgroup of $J$, which is contained in some conjugate of $K$.  Without loss of generality, we may assume $g^n$ lies in $K$.  Therefore, $g^n \sim k^*$ and we have that 
		\[g \sim g^n \sim k^* \sim c^*.\] Hence, in either case, we get that $d(g,c^*) \leq 3$.  Therefore, if $J$ is not a Frobenius group, we get that the diameter of the commuting graph is at most 6, which contradicts our initial assumption that the diameter is at least 7.
		
		Therefore, we must have that $J$ is a Frobenius group with complement $K$.  Since $J/F \cong K$, we get that $J/F$ is cyclic.  Thus, $\text{Aut}(J/F)$ is abelian, and we have that 
		\[ \normalizer{G/F}{J/F}/\centralizer{G/F}{J/F} = (G/F)/(J/F) \cong G/J \]
		is abelian.  Since $G/J$ is abelian, we have that the derived length of $G$ is at most 3.  Since $G$ is not abelian, the derived length is at least 2.  However, if the derived length is 2, then we have already shown in Lemma \ref{d2} that the diameter of $\Gamma(G)$ will be at most 4, which is a contradiction.  So that gives us that the derived length of $G$ is 3, and we can write $G$ as $G=HG''$, where $H=M_G(G')$ is a relative system normalizer for $G'$ in $G$.  Note that $H \cap G' = M_1$ is an absolute system normalizer of $G'$.  Furthermore, for the appropriate choice of Sylow system of $G'$, we have that an absolute system normalizer $M_0$ of $G$ will be contained in $H$.  As $\lvert H \rvert = \lvert M_0 \rvert \cdot \lvert M_1 \rvert$, we can write $H$ as $H=M_0M_1$.  Therefore, we get $G = M_0M_1G''$.  Note that $F = \centerof{G'} \times G''$ by Lemma \ref{FitGroup}.  As $J$ is a Frobenius group and $G' \leq J$, we have that $G'$ is Frobenius and therefore $\centerof{G'} =1$.  Hence, $F = G''$.
		
		Note that as $G$ is not a Frobenius group and every element of $M_1^{\#}$ acts fixed-point freely on $F$, we have that there exists $x^* \in M_0^{\#}$ such that $\centralizer{F}{x^*}\neq 1$.  Consider $\langle x^* \rangle F$.  Note that there are $\lvert G: \normalizer{G}{\langle x^* \rangle F}\rvert = \lvert M_1: \normalizer{M_1}{\langle x^*\rangle F} \rvert = \lvert M_1 : \normalizer{M_1}{\langle x^* \rangle} \rvert$ conjugates of $\langle x^* \rangle F$ in $G$ and that there are $\lvert H: \normalizer{H}{\langle x^* \rangle} \rvert = \lvert M_1 : \normalizer{M_1}{\langle x^* \rangle} \rvert$ conjugates of $\langle x^* \rangle$ in $H$.  Therefore, each conjugate of $\langle x^* \rangle$ in $H$ intersects a distinct conjugate of $\langle x^* \rangle F$ in $G$.  Furthermore, every conjugate of $H$ will intersect each conjugate of $\langle x^* \rangle F$ nontrivially.  We have that $\centerof{\langle x^* \rangle F} \neq 1$.  So choose $x^* \in M_0^{\#}$ to be as above and fix $c^* \in \centerof{\langle x^* \rangle F}^{\#}$.
		
		Let $g\in G$.  As in the previous case, if we can find an integer $n < o(g)$ so that $\centralizer{F}{g^n} \neq 1$, then $g \sim g^n \sim f \sim c^*$, where $f \in \centralizer{F}{g^n}^{\#},$ and $d(g,c^*) \leq 3$.  Hence, we must have that $\centralizer{F}{g^n} = 1$ for all $n<o(g)$.  In this case, the order of $g$ must be coprime to the order of the Fitting subgroup.  Hence, $g$ lies in some conjugate of H; without loss of generality, we can say that $g \in H$.
		
		First, suppose that $J_0 = J \cap M_0 \neq 1$.  Therefore, $G'\lneq J$.  Note that as $J/F = F(G/F)$, we have from Lemma \ref{FitGroup} that 
		\[J/F= \centerof{G/F} \times \centerof{(G/F)'} = \centerof{G/F} \times G'/F \neq G'/F.\]
		Therefore, $\centerof{G/F}\neq 1$.  Note that as $H \cong G/F$, we therefore have that $\centerof{H}\neq 1$.  Let $z \in \centerof{H}^{\#}$.  Therefore, we have that 
		\[g \sim z \sim x^* \sim c^*.\]
		Hence, $d(g,c^*) \leq 3$.  In both cases, we get that $d(g,c^*) \leq 3$ and so the diameter of $\Gamma(G)$ is at most 6, a contradiction.
		
		That leaves the case when $J_0 = 1$ and so $J = G'$.  However, from Lemma \ref{FitGroup}, we have that
		\[J/F = \centerof{G/F} \times G'/F = G'/F,\]
		which means that $\centerof{G/F}=1$.  Hence, we have that $\centerof{H}=1$. Suppose that $x \in M_0$ has prime order.  By Lemma \ref{p9}, we have that if $\centralizer{F}{x} = 1$, then $x \in J$.  However, $M_0 \cap J = 1$, which gives us a contradiction.  Thus, every element of prime order in $M_0$ centralizes some element in $F^{\#}$.
		
		Therefore, we can choose any such element to act as $x^*$ as in the previous case.  To ensure shortest path between $M_0$ and $M_1$, we will therefore choose $x^* \in M_0$ so that $\centralizer{M_1}{x^*} \neq 1$.  As before, note that each conjugate of $\langle x^* \rangle$ in $H$ corresponds to a unique conjugate of $\langle x^* \rangle F$ in $G$.  Fix $c^* \in \centerof{\langle x^*\rangle F}^{\#}$.
		
		By Lemma \ref{p9}, as $\centralizer{F}{g^n} = 1$ for all $n < o(g)$, we have that $g^n \in J = G'$ whenever $g^n$ has prime order.  More specifically, we get that $g^n$ lies in some conjugate of $M_1$.  Without loss of generality, $g^n \in M_1$.  Since this holds for all integers $n$ such that $o(g^n)$ is prime, we get that every prime dividing $o(g^n)$ also divides $\lvert M_1 \rvert$.  Hence, $g$ is a $\pi(M_1)$-element and so $g \in H$.  If $g \in M_1$, we have that \[g \sim y \sim x^* \sim c^*,\] where $y \in \centralizer{M_1}{x^*}^{\#}$.  Thus, $d(g,c^*) \leq 3$.  Thus, we must have that $g$ is a $\pi(M_1)$-element that does not lie in $M_1$.  Thus, there exists a prime $p$ dividing the order of $g$ such that the $p$-part of $g$ does not lie in $M_1$.  In particular, this gives us that $M_1$ is not a Hall subgroup of $H$.  Note that the absolute system normalizers of $H$ are the conjugates of $M_0$ contained in $H$.  Thus, we have that $H$ is a solvable $A$-group of derived length 2 and that $M_1 = H'$ is not a Hall subgroup of $H$; furthermore, we have that $g$ is an element of $H$ such that the $p$-part of $g$ does not lie in $M_1 = H'$ or any absolute system normalizer of $H$.  Hence, by Lemma \ref{need}, we have that $\centralizer{M_0}{g} \neq 1$.  Thus, we can find an element $x$ of prime order in $\centralizer{M_0}{g}^{\#}$ so that \[g \sim x \sim f \sim c^*,\] where $f \in \centralizer{F}{x}^{\#}$.  Therefore, $d(g,c^*) \leq 3$.  This means that in all cases, $d(g,c^*) \leq 3$.  Hence, we have that the diameter of the commuting graph is at most 6, contradicting our original hypothesis that the diameter was at least 7.  Therefore, we have that the diameter of the commuting graph will be at most 6.
	\end{proof}
	
	Having proven that the diameter of the commuting graph of a solvable group with a trivial center is at most 6, we are now ready to prove Theorem \ref{main}.
	
	\begin{proof}[Proof of Theorem \ref{main}]
		Let $G$ be a solvable $A$-group such that $G/Z$ is neither a Frobenius nor 2-Frobenius group.  Since $G/Z$ is neither a Frobenius nor 2-Frobenius group, we have that $\Gamma(G/Z)$ is connected.  Theorem \ref{dn} gives us that the diameter of $\Gamma(G/Z)$ is at most 6.  From \cite{REU}, we have that the diameter of $\Gamma(G)$ is the same as the diameter of $\Gamma(G/Z)$.  Thus, $\Gamma(G)$ is connected with diameter at most 6.
	\end{proof}
	
	Lastly, we provide an example of a solvable $A$-group with a connected commuting graph that has a diameter of 6, showing that the bound in Theorem \ref{main} cannot be improved.  The example comes from the Small Groups Library, SmallGroup(1500,115).  This group was investigated by Giudici and Pope in \cite{GiuPope} as part of their investigation of the commuting graph of solvable groups in the Small Groups Library.  Let $G$=SmallGroup(1500,115).  Note that the center of this group is trivial while the derived subgroup, which has order 375, has a nontrivial center.  The 2-elements in this group do not commute with any of the elements in the Fitting subgroup, which is the Sylow 5-subgroup of $G$.  Furthermore, the generators of order 4 do not commute with any 3-elements.  Hence, the distance between any element of order 4 and any nontrivial element in the center of the derived subgroup will be at most 3.  Thus, the maximum distance between any two elements of $G$ will be exactly 6.

\end{document}